\documentclass{article}

\usepackage{ulem}
\usepackage{amssymb,amsbsy,amsmath,amsfonts,amssymb,amscd,amsthm}
\usepackage{xcolor}
\usepackage{graphicx} 
\usepackage{multicol}
\graphicspath{{./}{figures/}} 
\usepackage{tikz}
\usetikzlibrary{matrix}
\newcommand{\R}{\mathbb{R}}

\newcommand{\N}{\mathbb{N}}
\newcommand{\T}{\mathbb{T}}
\newcommand{\EE}{\mathbb{E}}

\newcommand{\mA}{\mathcal A}
\newcommand{\mD}{\mathcal D}

\newcommand{\mX}{\mathcal X}
\newcommand{\mY}{\mathcal Y}

\newcommand{\beq}{\begin{equation}}
\newcommand{\eeq}{\end{equation}}
\def\a{\alpha}

\def\d{\delta}

\def\l{\lambda}

\def\r{\rho}
\def\s{\sigma}

\newcommand{\supp}{{\rm supp}}
\newcommand{\diver}{{\rm div}}
\def\pd{\partial}
\def\half{\frac{1}{2}}

\newcommand{\bphi}{\overline \phi}
\newcommand{\uphi}{\underline\phi}
\newcommand{\bpsi}{\bar \psi}

\newcommand{\sqi}{\frac{1}{\sigma^2}}
\newtheorem{theorem}{Theorem}[section]

\newtheorem{proposition}[theorem]{Proposition}

\numberwithin{equation}{section}

\title{A quadratic Mean Field Games model for the Langevin  equation}
\date{} 
\author{Fabio Camilli}
\begin{document}
\maketitle
\begin{abstract}
We consider a Mean Field Games model where the  dynamics of the agents is given by a controlled Langevin equation and  the cost is quadratic.  A    change of variables, introduced in \cite{gll}, transforms the Mean Field Games system into a  system  of two coupled kinetic Fokker-Planck equations. We prove an existence result for the latter system, obtaining consequently existence of a solution for the   Mean Field Games system. 
\end{abstract}
	
\noindent
{\footnotesize \textbf{AMS-Subject Classification:} 35K40, 91A13}.

\section{Introduction}
The Mean Field Games (MFG in short) theory concerns the study of differential games with a large number of rational, indistinguishable agents and the characterization of the corresponding Nash equilibria.
In the original model introduced in \cite{hcm,ll}, an agent can typically  act   on its  velocity (or other first order dynamical quantities) via a control variable. Mean Field Games where   agents control the acceleration have been recently proposed in \cite{ammt,bc,cm}.\par 
A prototype of stochastic process involving acceleration is given by  the Langevin diffusion process, which can be formally defined  as
\begin{equation}\label{eq:int_Langevin}
\ddot{X}(t)=-b(X(t))+\s\dot B(t),
\end{equation}
where   $\ddot{X}$ is the second  time derivative  of the stochastic process $X$,  $B$  a 
Brownian motion and $\s$ a positive parameter. The solution of  \eqref{eq:int_Langevin} can be rewritten as a Markov process $(X,V)$ solving 
\begin{equation*} 
\left\{
\begin{array}{ll}
\dot X(t)=V(t),\\
\dot V(t)=-b(X(t))+\s\dot B(t).
\end{array}
\right.
\end{equation*}
The probability density function of  the previous process  satisfies the kinetic Fokker-Planck equation
\begin{equation*}
\pd_t p-\frac{\s^2}{2}\Delta_v p-b(x)\cdot D_v p+v\cdot D_x p=0\qquad \text{in}\quad (0,\infty)\times \R^d\times\R^d. 
\end{equation*}
The previous equation, in the case  $b\equiv 0$, was first studied by Kolmogorov \cite{K} who provided  an explicit formula for its  fundamental solution. Then considered by H\"ormander \cite{H} as  motivating example for the general theory of the hypoelliptic operators (see also \cite{AM,bou,LP}).  \par
We consider a Mean Field Games model where the dynamics of the single agent is given by a controlled Langevin diffusion process, i.e 
\begin{equation}\label{eq:int_Langevin_controlled}
\left\{
\begin{array}{ll}
\dot X(s)=V(s),\, &s\ge t\\
\dot V(s)=-b(X(s))+\a(s)+\s\dot B(s)&s\ge t\\
X(t)=x,\,V(t)=v
\end{array}
\right.
\end{equation}
In \eqref{eq:int_Langevin_controlled},    the control law $\a:[t,T]\to\R^d$, which is a progressively measurable process with respect to a fixed filtered probability space such that $\EE[\int_t^T|\a(t)|^2dt]<+\infty$, is chosen to {\it maximize} the functional
\begin{align*} 
J(t,x,v;\a)= \EE_{t,(x,v)}\Big\{\int_t^T& \left[f(X(s),V(s),m(s))-\half |\a(s)|^2\right] ds\\
&+u_T(X(T),V(T))\Big\},
\end{align*}
where $m(s)$ is the distribution of the agents at time $s$. Let $u$ the value function associated with the previous control problem, i.e.
\[
u(t,x,v)=\sup_{\a\in \mA_t}\{J(t,x,v;\a)\}
\]
where $\mA_t$ is the the set of the control laws.
 Formally, the couple $(u,m)$ satisfies the MFG system (see \cite[Section 4.1]{ammt} for more details) 
\begin{equation}\label{eq:int_MFG}
\left\{
\begin{array}{ll}
\pd_t u+\frac{\s^2}{2}\Delta_v u-b(x)\cdot D_v u+v\cdot D_x u+\half |D_vu|^2=-f(x,v,m)\\[4pt]
\pd_t m-\frac{\s^2}{2}\Delta_v m-b(x)\cdot D_v m+v\cdot D_x m+\diver_v(mD_vu)=0\\[4pt]
m(0,x,v)=m_0(x,v),\quad u(T,x,v)=u_T(x,v).
\end{array}
\right.
\end{equation}
for $(t,x,v)\in (0,T)\times\R^d\times\R^d$. The first equation is a  backward Hamilton-Jacobi-Bellman equation, degenerate in the $x$-variable and  with a quadratic Hamiltonian in the $v$ variable, and the second equation is forward kinetic Fokker-Planck equation. In the standard setting, MFG systems with quadratic Hamiltonians has been extensively considered in literature both as a reference model for the general theory and also since, thanks to the Hopf-Cole change of variable, the nonlinear Hamilton-Jacobi-Bellman equation can  be transformed into a linear equation, allowing to use all the tools developed for this type of problem  (see for example \cite{gm,Gomesbook,G,gll,ll,ugs}). Recently, a similar procedure has been  used for ergodic  hypoelliptic MFG with quadratic cost in  \cite{fgt} and  for a flocking model involving kinetic equations  in \cite[Section 4.7.3]{cd}.\\
We study    \eqref{eq:int_MFG} by means of a change of variable   introduced in   \cite{G,gll} for the standard case. By defining the new unknowns $\phi =e^{u/{\s^2}}$ and $\psi =me^{-u/ {\s^2}}$,
the system \eqref{eq:int_MFG} is transformed into a   system of two kinetic Fokker-Planck equations
\begin{equation}\label{eq:int_MFG_kinetic}
\left\{
\begin{array}{ll}
\pd_t \phi+\frac{\s^2}{2}\Delta_v \phi-b(x)\cdot D_v \phi+v\cdot D_x \phi=-\sqi  f(x,v,\psi\phi)\phi\\[4pt]
\pd_t \psi-\frac{\s^2}{2}\Delta_v \psi-b(x)\cdot D_v \psi+v\cdot D_x \psi=\sqi f(x,v,\psi\phi)\psi\\[4pt]
\psi(0,x,v)=\frac{m_0(x,v)}{\phi(0,x,v)},\quad \phi(T,x,v)=e^{\frac{ u_T(x,v)}{ \s^2}}.
\end{array}
\right.
\end{equation}
for $(t,x,v)\in (0,T)\times\R^d\times\R^d$. In the previous problem, the coupling between the two equations is only in the source terms.
Following \cite{G}, we prove existence of a (weak) solution to \eqref{eq:int_MFG_kinetic} by showing the convergence of an iterative scheme defined,  starting from $\psi^{(0)}\equiv 0$, by solving alternatively  the  backward problem
\begin{equation}\label{eq:iterate_phi}
\left\{
\begin{array}{ll}
\pd_t \phi^{(k+\half)} +\frac{\s^2}{2}\Delta_v \phi^{(k+\half)}&\hskip-16pt -b(x)\cdot D_v \phi^{(k+\half)}+v\cdot D_x \phi^{(k+\half)}\\[4pt]
&=-\sqi f(\psi^{(k)}\phi^{(k+\half)})\phi^{(k+\half)}\\[6pt]
\phi^{(k+\half)} (T,x,v)=e^{\frac{ u_T(x,v)}{ \s^2}},
\end{array}
\right.
\end{equation}
and the forward one
\begin{equation}\label{eq:iterate_psi}
\left\{
\begin{array}{ll}
\pd_t \psi^{(k+1)}-\frac{\s^2}{2}\Delta_v \psi^{(k+1)}&\hskip-34pt -b(x)\cdot D_v \psi^{(k+1)}+v\cdot D_x \psi^{(k+1)}\\[4pt]
&=
\sqi f(\psi^{(k+1)}\phi^{(k+\half)})\psi^{(k+1)}\\[6pt]
\psi^{(k+1)}(0,x,v)=\frac{m_0(x,v)}{\phi^{(k+\half)}(0,x,v)}.
\end{array}
\right.
\end{equation}
 We show that the resulting   sequence    $(\phi^{(k+\half)},\psi^{(k+1)})$, $k\in\N$, monotonically converges to the solution of \eqref{eq:int_MFG_kinetic}. Hence, by the inverse change of variable 
\begin{equation}\label{eq:change}
	  u=\frac{ \ln(\phi)}{\s^2},\qquad m=\phi\psi,
\end{equation} we  obtain a solution of the original problem \eqref{eq:int_MFG}. We have
 \begin{theorem}\label{thm:main}
 	The sequence $(\phi^{(k+\half)},\psi^{(k+1)})$ defined by \eqref{eq:iterate_phi}-\eqref{eq:iterate_psi} converges in $L^2([0,T]\times \R^d\times\R^d)$ and a.e. to a weak solution $(\phi,\psi)$ of \eqref{eq:int_MFG_kinetic}. 
 	 Moreover, the couple $(u,m)$ defined by \eqref{eq:change} is a weak solution to \eqref{eq:int_MFG}.
 \end{theorem}

The previous iterative procedure also  suggests a monotone numerical   method for the approximation of    \eqref{eq:int_MFG_kinetic}, hence for \eqref{eq:int_MFG}. Indeed, by approximating \eqref{eq:iterate_phi}	and \eqref{eq:iterate_psi} by finite differences and solving alternatively the resulting   discrete equations, we obtain an approximation of the  sequence $(\phi^{(k+\half)},\psi^{(k+1)})$. A corresponding procedure for the standard  quadratic MFG system was studied  in \cite{G}, where the convergence of the method is proved.	We plan to study the properties of the previous numerical procedure in a future work.

\section{Well posedness of the  kinetic Fokker-Planck system}
In this section, we study the existence of a solution to  system \eqref{eq:int_MFG_kinetic}. The  proof of the  result follows the strategy implemented in \cite[Section 2]{G} for the case of a standard   MFG system with quadratic Hamiltonian    and relies on the results for   linear kinetic Fokker-Planck equations   in \cite[Appendix A]{d}.\\
 We fix the assumptions we will assume in all the paper. 
The vector field $b: \R^{d} \to\R^d$ and the coupling cost $f:\R^d
\times \R^d\times\R\to\R$  are assumed  to satisfy
\begin{eqnarray*} 
&b\in L^\infty(\T^{d}), \\
&\text{$f\in L^\infty(\R^d\times\R^d\times\R)$,   $f\le 0$ and $f(x,v,\cdot)$ strictly decreasing.}
\end{eqnarray*}
Moreover, the diffusion coefficient  $\s$ is   positive and   the initial and terminal data satisfy
\begin{equation}\label{hyp:i_m}
	\begin{aligned} 
		\,m_0\in L^\infty(\R^d\times\R^d),\, m_0\ge 0,\,\iint m_0(x,v)dxdv=1,\\
		\text{and $\exists\, R_0>0$ s.t. $\supp\{m_0\}\subset \R^d\times B(0,R_0)$}
	\end{aligned}
\end{equation}
and
\begin{equation}\label{hyp:i_u}
\begin{aligned} 
u_T\in C^0(\R^d\times\R^d)\,\text{ and $\exists\,C_0$, $C_1>0$ s.t. $\forall (x,v)\in \R^d\times\R^d$}\\
-C_0(|v|^2+|x|)-C_0\le u_T(x,v)\le -C_1(|v|^2+|x|)+C_1.
\end{aligned}
\end{equation} 
Note that \eqref{hyp:i_u} implies that $e^{u_T/\s^2}\in L^\infty(\R^d\times\R^d)\cap L^2(\R^d\times\R^d)$.
We denote  with $(\cdot,\cdot)$ the scalar product in $L^2([0,T]\times \R^d\times\R^d)$  and with $\langle\cdot,\cdot\rangle$   the pairing between $\mX=L^2([0,T]\times\R^d_x;H^1(\R^d_v))$ and its dual $\mX'=L^2([0,T]\times\R^d_x;H^{-1}(\R^d_v))$. We define the following functional space
\begin{equation*}
\mY=\left\{g\in L^2([0,T]\times\R^d_x, H^1(\R^d_v)),\partial_t g+v\cdot D_x g\in L^2([0,T]\times\R^d_x, H^{-1}(\R^d_v))\right\}
\end{equation*}
and  we set $\mY_0=\{g\in \mY:\,g\ge 0 \}$.  If $g\in \mY$, then it admits   (continuous)  trace values $g(0,x,v)$, $g(T,x,v)\in L^2(\R^d\times\R^d)$
(see \cite[Lemma A.1]{d}) and therefore   the initial/terminal conditions for \eqref{eq:int_MFG_kinetic} are well defined in $L^2$ sense. 
We first prove the well posedness of problems \eqref{eq:iterate_phi} and \eqref{eq:iterate_psi}.
\begin{proposition}\label{prop:well_posed_eq_phi}
We have
\begin{itemize}
\item[(i)] For any $\psi\in\mY_0$, there exists a unique solution $\phi\in\mY_0$ to	
\begin{equation}\label{eq:equation_phi}
\left\{
\begin{array}{ll}
\pd_t \phi+\frac{\s^2}{2}\Delta_v \phi-b(x)\cdot D_v \phi+v\cdot D_x \phi=-\sqi f(x,v,\psi\phi)\phi\\[4pt]
\phi(T,x,v)=e^{\frac{ u_T(x,v)}{ \s^2}}.
\end{array}
\right.
\end{equation}
Moreover, $\phi\in L^\infty([0,T]\times \R^d\times\R^d)$ and, for any  $R>0$, there exist $\d_R\in\R$ and $\rho>0$ such that
  \begin{equation}\label{eq:lower_bound_phi}
	\phi(t,x,v)\ge C_R:= e^{\sqi(\d_R-\rho T)}\quad\forall t\in [0,T],\, (x,v)\in B(0,R)\subset \R^d\times\R^d.
\end{equation}
\item[(ii)] Let $\Phi: \mY_0\to\mY_0$ be the map which associates	to $\psi$ the unique solution of \eqref{eq:equation_phi}.
	Then, if $\psi_2 \le\psi_1$, we have $\Phi(\psi_2)\ge \Phi(\psi_1)$.
\end{itemize}
\end{proposition}
\begin{proof}
Fixed $\psi\in\mY_0$, consider the map $F=F(\varphi)$ from $L^2([0,T]\times \R^d\times\R^d)$ into itself 
that associates  with $\varphi$ the weak solution $\phi\in L^2([0,T]\times \R^d\times\R^d)$ of the linear problem
\begin{equation}\label{eq:equation_phi_1}
\left\{
\begin{array}{ll}
\pd_t \phi+\frac{\s^2}{2}\Delta_v \phi-b(x)\cdot D_v \phi+v\cdot D_x \phi=-\sqi f(\psi\varphi)\phi\\[4pt]
\phi(T,x,v)=e^{\frac{ u_T(x,v)}{ \s^2}}.
\end{array}
\right.
\end{equation}
By \cite[Prop. A.2]{d},   $\phi$ belongs to $\mY$ and  it coincides with  the unique solution of  \eqref{eq:equation_phi_1} in this space. Moreover, the following estimate 
\begin{equation}\label{eq:energy_est}
\|\phi\|_{L^2([0,T]\times\R^d_x;H^1(\R^d_v))}+\|\pd_t\phi +v\cdot D_x\phi\|_{L^2([0,T]\times\R^d_x;H^{-1}(\R^d_v))}\le C
\end{equation}
holds for some constant $C$ which depends only on $\|e^{u_T/\s^2}\|_{L^2} $, $\|f\|_{L^\infty}$ and $\s $. Hence
$F$ maps $B_C$, the closed ball of radius $C$ of $L^2([0,T]\times \R^d\times\R^d)$, into itself.\\
To show that the map $F$ is continuous on $B_C$, consider    $\{\varphi_n\}_{n\in\N},\,\varphi\in L^2([0,T]\times \R^d\times\R^d)$ such that
$\|\varphi_n-\varphi\|_{L^2}\to 0$ and set  $\phi_n=F(\varphi_n)$. Then $\phi_n  \in \mY$, and, by the estimate
\eqref{eq:energy_est}, we get that, up to a subsequence, there exists $\bphi\in \mY$ such that
$\phi_n\to \bphi$, $D_v\phi_n\to D_v\bphi$ in $L^2([0,T]\times \R^d\times\R^d)$, $\pd_t\phi_n +v\cdot D_x\phi_n \to \pd_t\bphi_n +v\cdot D_x\bphi_n$ in $L^2([0,T]\times\R^d_x;H^{-1}(\R^d_v))$. Moreover $\varphi_n\to\varphi$ almost everywhere.  By the definition of weak solution to \eqref{eq:equation_phi_1},  we have that
\begin{equation}\label{eq:weak}
\langle\pd_t \phi_n+v\cdot D_x\phi_n,w \rangle - 
\frac{\s^2}{2}(D_v\phi_n,D_vw) -(b\cdot D_v\phi_n,w)=(-\frac{1}{\s^2}\phi_n F(\varphi_n \psi),w),
\end{equation}
for any $w\in \mD([0,T]\times \R^d\times\R^d)$, the space of infinite differentiable functions with compact support in $[0,T]\times \R^d\times\R^d$. 
Employing weak convergence for left hand side of \eqref{eq:weak} and the Dominated Convergence Theorem for the right hand one, we get for $n\to\infty$
\[
\langle\pd_t \bphi+v\cdot D_x\bphi,w \rangle - 
\frac{\s^2}{2}(D_v\bphi ,D_vw) -(b\cdot D_v\bphi,w)=(-\bphi  F(\varphi  \psi),w)
\]
for any $w\in \mD([0,T]\times \R^d\times\R^d)$. Hence $\bphi=F(\varphi)$ and $F(\varphi_n)\to F(\varphi)$ for $n\to\infty$ in $L^2([0,T]\times \R^d\times\R^d)$. 
The compactness of the map $F$ in $L^2([0,T]\times \R^d\times\R^d)$ follows by  the compactness of the set of the solutions to \eqref{eq:equation_phi_1}, see \cite[Theorem 1.2]{cep}. We conclude, by Schauder's Theorem,  that there exists a fixed-point of  the map $F$ in $L^2$, hence in $\mY$,  and  therefore a   solution to the nonlinear parabolic equation \eqref{eq:equation_phi}.\par
Observe that, if $\phi$ is a solution of
\eqref{eq:equation_phi}, then $\tilde \phi=e^{\l t}\phi$ is a solution of
\begin{equation}\label{eq:equation_phi_equiv}
\pd_t \tilde\phi+\frac{\s^2}{2}\Delta_v \tilde\phi-b(x)\cdot D_v \tilde\phi+v\cdot D_x\tilde \phi-\l\tilde \phi=-\sqi  f(e^{-\l t}\psi\tilde\phi)\tilde\phi
\end{equation}
with the corresponding final condition. In the following, we assume that $\lambda>0$.
To show that $\phi$ is non negative, we will exploit the following property (see \cite[Lemma A.3]{d}): given $\phi\in \mY$ and defined $\phi^{\pm} =\max(\pm \phi ,0)$, then $\phi^\pm\in \mX$ and
\begin{equation}\label{eq:property_phi_minus}
\langle \pd_t \phi+v\cdot D_x\phi,\phi^-\rangle=\half\left(\iint|\phi(0,x,v)^-|^2dxdv-\iint|\phi(T,x,v)^-|^2dxdv\right).
\end{equation}
Let $\phi$ be a solution of \eqref{eq:equation_phi_equiv}, multiply the equation by $\phi^-$ and integrate. Then, since $\phi(T,x,v)$ is non negative, by \eqref{eq:property_phi_minus} we get
\begin{align*}
	-\sqi (\phi f(e^{\l t}\phi \psi),\phi^-) =\langle\pd_t \phi+v\cdot D_x\phi,\phi^-\rangle-\\
	\frac{\s^2}{2}(D_v\phi,D_v\phi^-) -(b\cdot D_v\phi,\phi^-) -\l (\phi,\phi^-) =\\
	\half \iint|\phi(0,x,v)^-|^2dxdv+\frac{\s^2}{2}(D_v\phi^-,D_v\phi^-) +  
	\l (\phi^-,\phi^-) \ge \\ 
	 \l  (\phi^-,\phi^-),
\end{align*}
where it has been exploited that, by integration by parts,   $(b\cdot D_v\phi,\phi^-)=0$. Since $f\le 0$ and therefore
\[
-(\phi f(e^{\l t}\phi \psi),\phi^-) =(\phi^- f(e^{\l t}\phi \psi),\phi^-) \le 0,
\]
we get  $(\phi^-,\phi^-)\equiv 0$ , hence  $\phi\ge0$ .

To prove the uniqueness of the solution to \eqref{eq:equation_phi}, consider two solutions $\phi_1$, $\phi_2$ of \eqref{eq:equation_phi_equiv} and set $\bphi=\phi_1-\phi_2$. Multiplying the equation for $\bphi$    by $\bphi$, integrating and using $\bphi(x,v,T)=0$, we get
\begin{equation}\label{eq:uniq_phi}
	\begin{aligned}
		-\sqi(  f(e^{-\l t}\psi\phi_1)\phi_1- f(e^{-\l t}\psi\phi_2)\phi_2, \phi_1-\phi_2) =\langle\pd_t \bphi+v\cdot D_x\bphi,\bphi\rangle -\\
		\frac{\s^2}{2}(D_v\bphi,D_v\bphi) -(b\cdot D_v\bphi,\bphi) -\l (\bphi,\bphi) =\\
		-\half\iint|\bphi(x,v,0)|^2dxdv-\frac{\s^2}{2}(D_v\bphi,D_v\bphi)   -\l (\bphi,\bphi) \le 
		-\l(\phi_1-\phi_2,\phi_1-\phi_2)  
	\end{aligned}
\end{equation}
and, by the strict monotonicity of $f$, we conclude that
 $\phi_1=\phi_2$ .\par

To  prove that $\phi$ is bounded from above, we observe that the function $\bphi(t,x,v)=e^{C_1+(T-t)\|f\|_{\infty}/\s^2}$, where $C_1$ as in \eqref{hyp:i_u}, is a supersolution of the linear problem \eqref{eq:equation_phi_1} for any $\varphi\in L^2([0,T]\times \R^d\times\R^d)$, i.e.
$\phi(T,x,v)\ge e^{u_T(x,v)/\s^2}$ and 
\[\pd_t \bphi+\frac{\s^2}{2}\Delta_v \bphi-b(x)\cdot D_v \bphi+v\cdot D_x \bphi\le -\sqi f(\psi\varphi)\bphi.\]
By the Maximum Principle (see \cite[Prop. A.3 (i)]{d}), we get that $\bphi\ge \phi$, where $\phi$ is the solution of \eqref{eq:equation_phi_1}. Since the previous property holds for any $\varphi\in L^2([0,T]\times \R^d\times\R^d)$, we conclude that   $\bphi\ge \phi$, where $\phi$ is the solution of the nonlinear problem \eqref{eq:equation_phi}.\\
A similar argument allows to show that $\underline\phi(x,v,t)=e^{(-C_0(|v|^2+|x|+1)-\rho(T-t))/\s^2}$, where $C_0$ as in \eqref{hyp:i_u} and $\rho$ sufficiently large, is a subsolution of \eqref{eq:equation_phi_1} for any $\varphi\in L^2([0,T]\times \R^d\times\R^d)$. Indeed, replacing $\underline\phi$ in the equation, we get that the inequality
\begin{align*}
	\pd_t \uphi+\frac{\s^2}{2}\Delta_v \uphi-b(x)\cdot D_v \uphi+v\cdot D_x \uphi=\\
	=\frac{\uphi}{\s^2}\left(\rho- C_0d \s^2+ 2C_0^2 \s^2|v|^2+2C_0b(x)\cdot v-C_0v\cdot \frac{x}{|x|}\right)\ge \\-\frac {1 }{\s^2}  f(\psi\varphi)\uphi
\end{align*}
is satisfied for $\r$ large enough and, moreover, $\uphi\le e^{u_T(x,v)/\s^2}$. Hence $\underline \phi\le \phi$,  where $\phi$ is the solution of the nonlinear problem \eqref{eq:equation_phi}, and, from this estimate, we deduce \eqref{eq:lower_bound_phi}.\par
We finally prove the monotonicity of the map $\Phi$. 
Set $\phi_i=\Phi(\psi_i)$, $i=1,2$, and consider the equation satisfied by $\bphi=e^{\l t}\phi_1-e^{\l t}\phi_2$, multiply it by $\bphi^+$ and integrate. Performing a computation similar to \eqref{eq:uniq_phi}, we get
\begin{align*}
	-\sqi (f(\phi_1\psi_1)\phi_1-f(\phi_2\psi_2)\phi_2,\bphi^+)\le - \l(\bphi^+,\bphi^+).
\end{align*}
Since, by monotonicity of $f$ and  non negativity of $\phi_i$, we have
\begin{align*}
	-(f(\phi_1\psi_1)\phi_1-f(\phi_2\psi_2)\phi_2,\bphi^+)= -(f(\phi_1\psi_1)(\phi_1-\phi_2),\bphi^+)-\\
	((f(\phi_1\psi_1) -f(\phi_2\psi_2))\phi_2,\bphi^+)\ge 0,
\end{align*}
  we get   $(\bphi^+,\bphi^+)=0$	 and therefore $\phi_1\le \phi_2$.
\end{proof}
We set 
$$\mY_R=\{\phi\in \mY_0:\phi\ge C_R\quad \forall (x,v)\in B(0,R),\,t\in [0,T] \},$$
where $C_R$ is defined as in \eqref{eq:lower_bound_phi}. 
\begin{proposition}\label{prop:well_posed_eq_psi}
Given $R>R_0$, where $R_0$ as in \eqref{hyp:i_m},  we have
	\begin{itemize}
		\item[(i)] For any $\phi\in\mY_R$, there exists a unique solution $\psi\in\mY_0$ to	
		\begin{equation}\label{eq:equation_psi}
		\left\{
		\begin{array}{ll}
		\pd_t \psi-\frac{\s^2}{2}\Delta_v \psi-b(x)\cdot D_v \psi+v\cdot D_x \psi=\sqi f(x,v,\psi\phi)\psi\\[4pt]
		\psi(0,x,v)=\frac{m_0(x,v)}{\phi(0,x,v)}.
		\end{array}
		\right.
		\end{equation}
		Moreover
		\begin{equation}\label{eq:upper_bound_psi}
		\psi(x,v,t)\le \frac{\|m_0\|_{L^\infty}}{ C_R}\qquad\forall t\in [0,T],\, (x,v)\in \R^d\times\R^d,
		\end{equation}
		where $C_R$ as in \eqref{eq:lower_bound_phi}.
		\item[(ii)] Let $\Psi: \mY_R\to\mY_0$ be the map which associates	
		with $\phi\in \mY_R$ the unique solution of \eqref{eq:equation_psi}.
		Then, if $\phi_2 \le \phi_1$, we have $\Psi(\phi_2)\ge \Psi(\phi_1)$.
	\end{itemize}	
\end{proposition}
\begin{proof}
First observe that, since  $R>R_0$,  then    $\psi(0,x,v)$ is well defined for $\phi\in\mY_R$.
The proof  of the first part of $(i)$ is very similar to the one of the corresponding result in Proposition \ref{prop:well_posed_eq_phi}, hence we only prove the bound \eqref{eq:upper_bound_psi}. 
If $\psi$ is a solution  of \eqref{eq:equation_psi}, then $\tilde \psi=e^{-\l t}\psi$ is a solution of
\begin{equation} \label{eq:equation_psi_equiv}
	\pd_t \tilde\psi-\frac{\s^2}{2}\Delta_v \tilde\psi-b(x)\cdot D_v \tilde\psi+v\cdot D_x \psi+\l\tilde\psi=\sqi f(x,v,e^{\l t}\tilde\psi\phi)\psi.
\end{equation}
Let $\psi$ be a solution of \eqref{eq:equation_psi_equiv},  set $\bpsi=\psi- e^{-\l t} \|m_0\|_{L^\infty}/ C_R$ and observe that $\bpsi(0)\le 0$. Multiply the equation
for  $\bpsi$ by $\bpsi^+$ and integrate to obtain
\begin{align*}
	(\psi f(e^{\l t}\psi\phi),\bpsi^+)=\\ \langle\pd_t\bpsi+v\cdot D_x\bpsi,\bpsi^+\rangle+\sqi(D_v\bpsi,D_v\bpsi^+)-(b(x)D_v\bpsi,\bpsi^+)+\l (\bpsi,\bpsi^+)\ge\\
	\iint|\bpsi^+(x,v,T)|^2dxdv+\l (\bpsi^+,\bpsi^+)\ge
	 \l(\bpsi^+,\bpsi^+).
\end{align*}
Since $\psi\ge 0$ and $f\le 0$, we have
\[(\psi f(e^{\l t}\psi\phi),\bpsi^+)\le 0\]
and therefore $\bpsi^+\equiv 0$. Hence the upper bound \eqref{eq:upper_bound_psi}.\\
Now we prove {\it (ii)}.  Set $\psi_i=\Psi(\phi_i)$, $i=1,2,$ and   $\bpsi=e^{-\l t}\psi_1-e^{-\l t}\psi_2$. Multiply the equation satisfied by $\bpsi$ by  $\bpsi^+$ and integrate. Since, by monotonicity and negativity of $f$, we have
\begin{align*}
	(f(e^{\l t}\phi_1\psi_1)\psi_1-f(e^{\l t}\phi_2\psi_2)\psi_2,\bpsi^+)=
	(f(e^{\l t}\phi_1\psi_1) (\psi_1-\psi_2),\bpsi^+)+\\
	(\psi_2(f(e^{-\l t}\phi_1\psi_1 ) -f(e^{-\l t}\phi_2\psi_2 )),\bpsi^+)\le 0.
\end{align*}
Then
\begin{align*}
	0\ge \langle\pd_t\bpsi+v\cdot D_x\bpsi,\bpsi^+\rangle+\sqi(D_v\bpsi,D_v\bpsi^+)-(b(x)D_v\bpsi,\bpsi^+)+\l (\bpsi,\bpsi^+)\ge\\
	\iint|\bpsi^+(x,v,T)|^2dxdv+\l (\bpsi^+,\bpsi^+)\ge 
	\l(\bpsi^+,\bpsi^+).
\end{align*}
Hence $\bpsi^+\equiv 0$ and therefore $\psi_1\le\psi_2$.

\end{proof}
\begin{proof}[Proof of Theorem \ref{thm:main}]
Given   $\psi^{(0)}\equiv 0$, consider the sequence $(\phi^{(k+\half)},\psi^{(k+1)})$, $k\in\N$, defined in \eqref{eq:iterate_phi}-\eqref{eq:iterate_psi}. It can rewritten as   
\begin{equation}\label{eq:iterative_seq}
\left\{\ 
\begin{array}{l}
\phi^{(k+\half)}=\Phi(\psi^{(k)})\\[6pt]
\psi^{(k+1)}=\Psi(\phi^{(k+\half)})
\end{array}
\right.
\end{equation}
where the maps $\Phi$, $\Psi$ are   as in Propositions \ref{prop:well_posed_eq_phi} and, respectively,  
\ref{prop:well_posed_eq_psi}. Observe that, by \eqref{eq:lower_bound_phi}, we have $\phi^{(k+\half)} \in\mY_R$ for $R>R_0$   and $\psi^{(k+1)}\ge 0$ for any $k$. Hence the sequence $(\phi^{(k+\half)},\psi^{(k+1)})$ is well defined.
We first prove by induction the monotonicity of the components of $(\phi^{(k+\half)},\psi^{(k+1)})$. By   non negativity of solutions to \eqref{eq:equation_psi}, we have  $\psi^{(1)}=\Phi(\phi^{(\half)})\ge 0$ and therefore $\psi^{(1)}\ge\psi^{(0)}$. Moreover, by the monotonicity of $\Phi$, $\phi^{(\frac 3 2)}=\Phi(\psi^{(1)})\le \Phi(\psi^{(0)})=\phi^{(\half)}$. Now assume that $\psi^{(k+1)}\ge \psi^{(k)}$. Then 
\[ \phi^{(k+\frac 3 2)}=\Phi(\psi^{(k+1)})\le \Phi(\psi^{(k)})=\phi^{(k+\half)}\]
and  
\[ \psi^{(k+2)}  =\Psi  (\phi^{(k+\frac 3 2)}) \ge  \Psi  (\phi^{(k+\half)})=\psi^{(k+1)},  \]
therefore the monotonicity of two sequences.\\
Since $\phi^{(k+\half)}\ge 0$ and, by \eqref{eq:upper_bound_psi}, for $k\to\infty$, the sequence $\psi^{(k+1)}\le \|m_0\|_{L^\infty}/C_R$,   $(\phi^{(k+\half)},\psi^{(k+1)})$ converges a.e. and in $L^2([0,T]\times \R^d\times\R^d)$ to a couple $(\phi,\psi)$. Taking into account the estimate \eqref{eq:energy_est}, the a.e. convergence of the two sequences and repeating an argument similar to the one employed for the continuity of the map $F$ in Proposition \ref{prop:well_posed_eq_phi}, we   get that the couple  $(\phi,\psi)$ satisfies, in weak sense, the first two equations in \eqref{eq:int_MFG_kinetic}. The terminal condition for $\phi$ is obviously satisfied, while the initial condition for $\psi$, in $L^2$ sense, follows  by convergence of $\phi^{(k+\half)}(0)$ to $\phi(0)$.\par
We now consider the couple $(u,m)$ given by the change of variable in \eqref{eq:change}. We first observe that, by \cite[Theorem 1.5]{bou}, we have $\pd_t\phi+v\cdot D_x\phi$, $D_v\phi$, $\Delta_v\phi\in L^2([0,T]\times \R^d\times\R^d)$ and a corresponding regularity for $\psi$. Taking into account the boundedness of $\phi$  and the estimate in \eqref{eq:lower_bound_phi}, we have that $u$, $\pd_t u+v\cdot D_x u$, $D_v u$, $\Delta_v u\in L^2_{loc}([0,T]\times \R^d\times\R^d)$. Hence we can write the equation for $u$ in weak form, i.e.
\[
	(\pd_t u+v\cdot D_x u,w )- 
	\frac{\s^2}{2}(D_vu,D_vw) -(b\cdot D_vu,w)+\half(|D_vu|^2,w)=-(   f(m),w),
\]
for any $w\in \mD([0,T]\times \R^d\times\R^d)$, with final datum in trace sense. In a similar way,  since $m$, $\pd_t m+v\cdot D_x m$, $D_v m$, $\Delta_v m\in L^2_{loc}([0,T]\times \R^d\times\R^d)$ and $m$ is locally bounded, we can   rewrite also the equation for $m$ in weak form, i.e.
\[
(\pd_t m+v\cdot D_x m,w )+
\frac{\s^2}{2}(D_vm,D_vw) -(b\cdot D_vm,w)-(mD_vu,Dw)=0,
\]
for any $w\in \mD([0,T]\times \R^d\times\R^d)$  with the initial datum in trace sense.
\end{proof}
\noindent{\bf Acknowledgements.}  The author  wishes to thank  Alessandro Goffi (Univ. di Padova)  and
Sergio Polidoro (Univ. di Modena e Reggio Emilia) for useful discussions.

	
\begin{flushright}
	\noindent \verb"fabio.camilli@uniroma1.it"\\
	Dip. di Scienze di base e  applicate per l'Ingegneria\\
	 Sapienza Universit\`{a} di Roma\\
	via A.Scarpa 14, 00161 Roma (Italy)	\\
	
\end{flushright}

\end{document}